\documentclass[dvips,12pt,a4paper]{amsart}

\usepackage[matrix,arrow,curve]{xy}
\usepackage{amsmath,amssymb}

\usepackage[dvips]{graphicx}

\def\C{{\mathbb C}}
\def\bL{{\mathbb L}}
\def\mD{{\mathcal D}}
\def\Z{{\mathbb Z}}

\def\bH{{\overline H}}
\def\bn{{\overline n}}

\newtheorem{proposition}{Proposition}[section]
\newtheorem{theorem}[proposition]{Theorem}

\newtheorem{lemma}[proposition]{Lemma}
\newtheorem{conjecture}[proposition]{Conjecture}

\newtheorem{remark}[proposition]{Remark}

\title[Quasihomogeneous Hilbert schemes]{The classes of the quasihomogeneous Hilbert schemes of points on the plane}

\thanks{The author is partially supported by the grants RFBR-10-01-00678, NSh-8462.2010.1, the Vidi grant of NWO and by the Moebius Contest Foundation for Young Scientists}

\author{A. Buryak}

\address{Faculty of Mechanics and Mathematics, Moscow State University, 119991 Moscow, Russia and\newline\indent
Department of Mathematics, University of Amsterdam, P.~O.~Box 94248, 1090 GE Amsterdam, The Netherlands}
\email{buryaksh@mail.ru, a.y.buryak@uva.nl}



\begin{document}

\begin{abstract}
In this paper we give a formula for the classes (in the Grothendieck ring of complex quasi-projective varieties) of irreducible components of $(1,k)$-quasi-homogeneous Hilbert schemes of points on the plane. We find a new simple geometric interpretation of the $q,t$-Catalan numbers. Finally, we investigate a connection between $(1,k)$-quasi-homogeneous Hilbert schemes and homogeneous nested Hilbert schemes. 
\end{abstract}

\maketitle

\section{Introduction}

The Hilbert scheme $(\C^2)^{[n]}$ of $n$ points in the plane $\C^2$ parametrizes the ideals $I\subset\C[x,y]$ of colength $n$: $dim_{\C}\C[x, y]/I=n$. There is an open dense subset of $(\C^2)^{[n]}$ that parametrizes the ideals associated with configurations of $n$ distinct points. The Hilbert scheme of $n$ points in the plane is a nonsingular, irreducible, quasiprojective algebraic variety of dimension $2n$ with a rich and much studied geometry, see \cite{Gottsche,Nakajima} for an introduction.

The cohomology groups of $(\C^2)^{[n]}$ were computed in \cite{Ellingsrud} and we refer the reader to the papers \cite{Costello,Lehn1,Lehn2,Li,Vasserot} for the description of the ring structure in the cohomology $H^*((\C^2)^{[n]})$. Let $\bn=(n_1,\ldots,n_k)$. The nested Hilbert scheme $(\C^2)^{[\bn]}$ parametrizes $k$-tuples $(I_1,I_2,\ldots,I_k)$ of ideals $I_j\subset\C[x, y]$ such that $I_j\subset I_h$ for $j<h$ and $dim_{\C}\C[x, y]/I_j=n_j$. In \cite{Cheah} J. Cheah studied smoothness and the homology groups of the nested Hilbert schemes $(\C^2)^{[\bn]}$.

There is a $(\C^*)^2$-action on $(\C^2)^{[n]}$ that plays a central role in this subject. The algebraic torus $T=(\C^*)^2$ acts on $\C^2$ by scaling the coordinates, $(t_1,t_2)(x,y)=(t_1x,t_2y)$. This action lifts to the $T$-action on the Hilbert scheme $(\C^2)^{[n]}$.

Let $T_{a,b}=\{(t^a,t^b)\in T|t\in\C^*\}$, where $a,b\ge 1$ and $gcd(a,b)=1$, be a one dimensional subtorus of $T$. Let $(\C^2)^{[n]}_{a,b}$ be the set of fixed points of the $T_{a,b}$-action on the Hilbert scheme $(\C^2)^{[n]}$. The variety $(\C^2)^{[n]}_{a,b}$ is smooth and parameterizes quasi-homogeneous ideals of colength $n$ in the ring $\C[x,y]$. Irreducible components of $(\C^2)^{[n]}_{1,1}$ were described in \cite{Iarrobino} and a description of the irreducible components of $(\C^2)^{[n]}_{a,b}$ for arbitrary $a$ and $b$ was obtained in \cite{Evain}.   

We denote by $K_0(\nu_{\C})$ the Grothendieck ring of complex quasiprojective varieties. The classes of the irreducible components of the Hilbert scheme $(\C^2)^{[n]}_{1,1}$ in $K_0(\nu_{\C})$ were computed in \cite{Iarrobino2}. 

Let $(\C^2)^{[\bn]}_{a,b}$ be the set of fixed points of the $T_{a,b}$-action on the nested Hilbert scheme $(\C^2)^{[\bn]}$. The dimensions of the irreducible components of $(\C^2)^{[(n,n+1)]}_{1,1}$ were computed in \cite{Cheah}. 

In this paper we generalize the result of \cite{Iarrobino2} and give a formula for the classes in $K_0(\nu_{\C})$ of the irreducible components of the variety $(\C^2)^{[n]}_{1,k}$ for an arbitrary positive $k$. As an application, we find an interesting combinatorial identity. We formulate a conjectural formula for the generating series of the classes $\left[(\C^2)^{[n]}_{a,b}\right]$. The combinatorics related to the action of the torus $T_{1,k}$ is very similar to the combinatorics of the $k$-parameter $q,t$-Catalan numbers and we find a new simple geometric interpretation of these numbers.

We also investigate a connection between $(1,k)$-quasi-homogeneous Hilbert schemes and homogeneous nested Hilbert schemes. We construct a natural map $\pi\colon(\C^2)^{[n]}_{1,k}\to (\C^2)^{[\bn]}_{1,1}$. We find a sufficient condition for the restriction of this map to an irreducible component to be an isomorphism. In particular, this condition is satisfied when $\bn=(n+1,n)$. Hence, we generalize the result from \cite{Cheah}, where the dimensions of the irreducible components in this case were computed.  

\subsection{Grothendieck ring of quasi-projective varieties}
Here we recall a definition of the Grothendieck ring $K_0(\nu_{\C})$ of complex quasi-projective varieties. It is the abelian group generated by the classes $[X]$ of all complex quasi-projective varieties $X$ modulo the relations:
\begin{enumerate}
\item if varieties $X$ and $Y$ are isomorphic, then $[X]=[Y]$;
\item if $Y$ is a Zariski closed subvariety of $X$, then $[X]=[Y]+[X\backslash Y]$.
\end{enumerate}  
The multiplication in $K_0(\nu_{\C})$ is defined by the Cartesian product of varieties: $[X_1]\cdot[X_2]=[X_1\times X_2]$. The class $\left[\mathbb A^1_{\C}\right]\in K_0(\nu_{\C})$ of the complex affine line is denoted by $\bL$.

\subsection{Description of the irreducible components of $(\C^2)^{[n]}_{a,b}$}
Let us recall a description of the irreducible components of the variety $(\C^2)^{[n]}_{a,b}$. Let $\C[x,y]^d_{a,b}\subset\C[x,y]$ be the subspace of quasihomogeneous polynomials of degree $d$ with respect to the action of $T_{a,b}$. Let $H=(d_0,d_1,\ldots)$ be a sequence of non-negative integers such that $\sum_{i\ge 0}d_i=n$. Let $(\C^2)^{[n]}_{a,b}(H)\subset(\C^2)^{[n]}_{a,b}$ be the set of points corresponding to quasihomogeneous ideals $I\subset\C[x,y]$ such that $\dim(\C[x,y]^i_{a,b}/(I\cap\C[x,y]^i_{a,b}))=d_i$. 
\begin{proposition}[\cite{Evain}]
If $(\C^2)^{[n]}_{a,b}(H)\ne\emptyset$, then $(\C^2)^{[n]}_{a,b}(H)$ is an irreducible component of $(\C^2)^{[n]}_{a,b}$. 
\end{proposition}

\subsection{Classes of the irreducible components of $(\C^2)^{[n]}_{1,k}$}
In this section we fix $k\ge 1$. For numbers $M,N\ge 0$ let $G(M,N)_q=\frac{\prod_{i=1}^{M+N}(1-q^i)}{\prod_{i=1}^M(1-q^i)\prod_{i=1}^{N}(1-q^i)}$. Let $\eta(H)$ be the largest $i$, such that $d_i=\left[\frac{i}{k}\right]+1$. We adopt the following conventions, $\eta(H)=-1$, if $H=(0,0,\ldots)$; $d_{-1}=0$. We introduce an auxiliary function $\tau$ defined by the following rule, $\tau(i)=1$, if $k\mid i+1$ and $\tau(i)=0$, if $k\nmid i+1$. We will prove the following statement.
\begin{theorem}\label{main theorem}
Let $H=(d_0,d_1,\ldots),n=\sum\limits_{i\ge 0}d_i$. If $(\C^2)^{[n]}_{1,k}(H)\ne\emptyset$, then
\begin{gather*}
\left[(\C^2)^{[n]}_{1,k}(H)\right]=\prod_{i\ge\eta}G(d_i-d_{i+1}+\tau(i),d_{i+1}-d_{i+1+k})_{\bL}.
\end{gather*}
\end{theorem}
\begin{remark}
We see that the classes of the irreducible components of $(\C^2)^{[n]}_{1,k}$ are polynomials in $\bL$. Moreover, all roots of these polynomials are the roots of unity. In the case of an arbitrary pair $(a,b)$, this is not true. For example, it is easy to compute that  
\begin{gather*}
\left[(\C^2)^{[12]}_{2,3}(1,0,1,1,1,1,2,1,1,1,1,0,1)\right]=1+3\bL+\bL^2.
\end{gather*}
\end{remark}

\subsection{Conjecture}
The following conjectural formula for the generating series of the classes $\left[(\C^2)^{[n]}_{a,b}\right]$ is based on computer calculations.
\begin{conjecture}
\begin{gather*}
\sum_{n\ge 0}\left[(\C^2)^{[n]}_{a,b}\right]t^n=\prod_{\substack{i\ge 1\\(a+b)\nmid i}}\frac{1}{1-t^i}\prod_{i\ge 1}\frac{1}{1-\bL t^{(a+b)i}}.
\end{gather*}
\end{conjecture}
Similar conjectural formulas for the generating series of the classes of some equivariant Hilbert schemes can be found in \cite{Gusein-Zade}.

\subsection{Definition of the $(q,t)$-Catalan numbers} 

A $k$-Dyck path is a lattice path from $(0,0)$ to $(kn,n)$ consisting of $(0,1)$ and $(1,0)$ steps, never going below the line $x=ky$ (see Figure~\ref{pic2}). Let $L^+_{kn,n}$ denote the set of these paths. For a $k$-path $\pi$ let $D'_{\pi}$ be the set of squares which are above $\pi$ and contained in the rectangle with vertices $(0,0)$, $(kn,0)$, $(kn,n)$ and $(0,n)$. The set $D'_{\pi}$ reflected with respect to the horizontal line is a Young diagram. We denote it by $D_{\pi}$. 

For a Young diagram $D$ and a box $s\in D$ let $a(s)$ denote the number of boxes in $D$ in the same column and strictly above $s$ and let $l(s)$ denote the number of boxes in $D$ in the same row and strictly right of $s$ (see Figure~\ref{pic1}).

\begin{figure}[h]
\begin{center}
\includegraphics{ex1.1}
\end{center}
\caption{}
\label{pic1}
\end{figure}

For a $k$-path $\pi$ let $area(\pi)$ be the number of full squares below $\pi$ and above the line $ky = x$, and let
\begin{gather*}
b_k(\pi)=|\{s\in D_{\pi}| ka(s)\le l(s)\le k(a(s) + 1)\}|.
\end{gather*}
An example is on Figure~\ref{pic2}.

\begin{figure}[h]
\begin{center}
\includegraphics{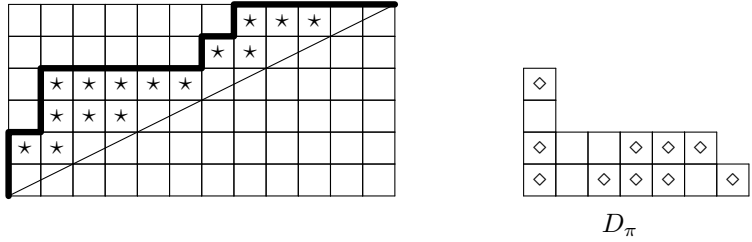}
\end{center}
\caption{A $2$-path $\pi$ with $area(\pi)=15$ and $b_k(\pi)=10$ (contributors to $b_k(\pi)$ are marked by $\diamond$, and those to $area(\pi)$ by $\star$).}
\label{pic2}
\end{figure}

The combinatorial $k$-parameter $(q,t)$-Catalan number is defined by the formula
\begin{gather*}
C_n^{(k)}(q,t)=\sum_{\pi\in L_{kn,n}^+}q^{b_k(\pi)}t^{area(\pi)}.
\end{gather*}
We refer the reader to the book \cite{Haglund} for another equivalent beautiful definitions of the $q,t$-Catalan numbers.

\subsection{$(q,t)$-Catalan numbers and the Hilbert schemes}

Let $V_{k,n}$ be the vector subspace of $\C[x,y]$ generated by the monomials $x^iy^j$ with $i+kj\le kn-k-1$. Let $(\C^2)^{[N](k,n)}$ be the subset of $(\C^2)^{[N]}$ that parametrizes ideals $I\subset \C[x,y]$ such that $I+V_{k,n}=\C[x,y]$. It is easy to see that $(\C^2)^{[N](k,n)}$ is an open subset of the variety $(\C^2)^{[N]}$.

\begin{theorem}\label{theorem:q,t-Catalan}
$$
\sum_{N\ge 0}\left[(\C^2)^{[N](k,n)}\right]t^N=(\bL t)^{\frac{kn(n-1)}{2}}C_n^{(k)}(\bL,\bL^{-1}t^{-1}).
$$
\end{theorem}

\subsection{Combinatorial identity}

We say that a sequence $H=(d_0,d_1,\ldots)$ is good if for any $i\ge \eta(H)$ we have $d_i-d_{i+1}+\tau(i)\ge 0$ and $d_{i+1}\le d_{i+1-k}$.

\begin{theorem}\label{theorem:combinatorial identity}
$$
\sum_{\{\text{good } H\}}\prod_{i\ge\eta}G(d_i-d_{i+1}+\tau(i),d_{i+1}-d_{i+1+k})_q q^{\chi(H)}t^{\sum d_i}=\prod_{i\ge 1}\frac{1}{1-qt^i},
$$
where 
\begin{multline*}
\chi(H)=\sum_{i\ge\eta}(d_i-d_{i+1}+\tau(i))\times\\\times\left(\frac{k}{2}(d_i-d_{i+1}+\tau(i)-1)+\sum_{j=1}^{k-1}(k-j)(d_{i+j}-d_{i+j+1}+\tau(i+j))\right).
\end{multline*}
\end{theorem}
In the case $k=1$ this identity was proved in \cite{Kirillov}.

\subsection{Homogeneous nested Hilbert schemes}\label{main results:nested}

Let $\bn=(n_1,n_2,\ldots,n_k)$, where $n_1,\ldots,n_k$ are non-negative integers such that $n_1\ge n_2\ge\ldots\ge n_k$. Let $\bH=(H_1,H_2,\ldots,H_k)$, where $H_i=(d_{i,0},d_{i,1},\ldots)$ and $\sum_{j\ge 0}d_{i,j}=n_i$. Let $(\C^2)^{[\bn]}_{a,b}(\bH)=\{(Z_1,\ldots,Z_k)\in(\C^2)^{[\bn]}|Z_i\in(\C^2)^{[n_i]}_{a,b}(H_i)\}$. Let $E(\bH)=\{i\in \Z_{\ge 0}|d_{1,i}=d_{2,i}=\ldots=d_{k,i}\}$, $n=\sum_{i=1}^k n_i$ and $H=(d_0,d_1,\ldots)$, where $d_{i+kj}=d_{i+1,j}, 0\le i<k, j\ge 0$. We will prove the following statement.
\begin{theorem}\label{theorem:nested}
Suppose that for any two numbers $i,j\in\Z_{\ge 0}\backslash E(\bH), i<j$, we have $j-i\ge 2$. Then the variety $(\C^2)^{[\bn]}_{1,1}(\bH)$ is isomorphic to $(\C^2)^{[n]}_{1,k}(H)$. 
\end{theorem}

\subsection{Organization of the paper}
In section \ref{section:cellular decomposition} we construct a cellular decomposition of the quasihomogeneous Hilbert scheme and reduce Theorem \ref{main theorem} to a combinatorial identity. In section \ref{section:bijection} we construct a bijection that is a generalization of the hook code from \cite{Iarrobino2}. The main result of this section is Proposition \ref{proposition:bijection}. Finally, in section \ref{section:proof of theorem} we apply it to conclude the proof of Theorem \ref{main theorem}. The proof of Theorem \ref{theorem:q,t-Catalan} is in section \ref{section:proof of Catalan}. We prove Theorem \ref{theorem:combinatorial identity} in section \ref{section:proof of combinatorial identity}. Section \ref{section:homogeneous nested} contains the proof of Theorem \ref{theorem:nested}. 

\subsection{Acknowledgments}
The author is grateful to S. M. Gusein-Zade for suggesting the area of research. The author is grateful to B. L. Feigin and A. N. Kirillov who noticed that the main result from \cite{Iarrobino2} can be generalized to the $(1,k)$-case and suggested the author to work on this problem. The author is grateful to S. Shadrin, M. Kazarian, S. Lando and A. Oblomkov for useful discussions. 

\section{Cellular decomposition of $(\C^2)^{[n]}_{1,k}$}\label{section:cellular decomposition}

In this section we reduce Theorem \ref{main theorem} to the combinatorial identity~\eqref{combinatorial identity} using a cellular decomposition of $(\C^2)^{[n]}_{1,k}$.

Consider the $T$-action on $(\C^2)^{[n]}$. Fixed points of this action correspond to monomial ideals in $\C[x,y]$. Let $I\subset\C[x,y]$ be a monomial ideal of colength $n$. Let $D_I=\{(i,j)\in\Z_{\ge 0}^2|x^iy^j\notin I\}$ be the corresponding Young diagram. We will use the following notations. For a Young diagram $D$ let
\begin{align*}
&r_l(D)=|\{(i,j)\in D|j=l\}|,\\
&c_l(D)=|\{(i,j)\in D|i=l\}|,\\
&diag^{a,b}_l(D)=|\{(i,j)\in D|ai+bj=l\}|,\\
&diag^{a,b}(D)=(diag^{a,b}_0(D),diag^{a,b}_1(D),diag^{a,b}_2(D),\ldots).
\end{align*}

Let $p\in(\C^2)^{[n]}$ be the fixed point corresponding to a Young diagram $D$. Let $R(T)=\Z[t_1,t_2]$ be the representation ring of $T$. Then the weight decomposition of $T_p(\C^2)^{[n]}$ is given by (see \cite{Ellingsrud})
\begin{gather}\label{weight decomposition}
T_p(\C^2)^{[n]}=\sum_{s\in D}\left(t_1^{l(s)+1}t_2^{-a(s)}+t_1^{-l(s)}t_2^{a(s)+1}\right).
\end{gather}
Obviously, the variety $(\C^2)^{[n]}_{1,k}$ is invariant under the $T$-action and contains all fixed points of the $T$-action on $(\C^2)^{[n]}$. Hence, the weight decomposition of $T_p(\C^2)^{[n]}_{1,k}$ is given by
\begin{gather*}
T_p(\C^2)^{[n]}_{1,k}=\sum_{\substack{s\in D\\l(s)+1=ka(s)}}t_1^{l(s)+1}t_2^{-a(s)}+\sum_{\substack{s\in D\\l(s)=k(a(s)+1)}}t_1^{-l(s)}t_2^{a(s)+1}.
\end{gather*}

Consider the $T_{1,\alpha}$-action on $(\C^2)^{[n]}_{1,k}$, where $\alpha$ is a positive integer. If $\alpha$ is big enough then the set of fixed points of the $T_{1,\alpha}$-action coincides with the set of fixed points of the $T$-action. For a fixed point $p\in(\C^2)^{[n]}_{1,k}$ let $C_p=\{z\in(\C^2)^{[n]}_{1,k}|\lim_{t\to 0,t\in T_{1,\alpha}}tz=p\}$. The variety $(\C^2)^{[n]}_{1,k}$ has a cellular decomposition with the cells $C_p$ (see \cite{B1,B2}). Therefore, the cells $C_p$ are isomorphic to affine spaces. It is easy to compute that if a point $p$ corresponds to a Young diagram $D$, then $\dim(C_p)=|\{s\in D|l(s)=k(a(s)+1)\}|$. Moreover, $p\in(\C^2)^{[n]}_{1,k}(H)\Leftrightarrow diag^{1,k}(D)=H$, where $H=(d_0,d_1,\ldots)$ is an arbitrary sequence of non-negative integers. 

Let $\mathcal D$ be the set of Young diagrams. We see that 
\begin{gather}\label{identity1}
\left[(\C^2)^{[n]}_{1,k}(H)\right]=\sum_{\substack{D\in\mathcal D\\diag^{1,k}(D)=H}}\bL^{|\{s\in D|l(s)=k(a(s)+1)\}|}.
\end{gather}
Therefore, Theorem \ref{main theorem} follows from the combinatorial identity:
\begin{gather}\label{original identity}
\sum_{\substack{D\in\mathcal D\\diag^{1,k}(D)=H}}q^{|\{s\in D|l(s)=k(a(s)+1)\}|}=\prod_{i\ge\eta}G(d_i-d_{i+1}+\tau(i),d_{i+1}-d_{i+1+k})_{q}.
\end{gather}
It is not hard to check that this identity is equivalent to the following identity
\begin{multline}\label{combinatorial identity} 
\sum_{\substack{D\in\mathcal D\\diag^{1,k}(D)=H}}q^{|\{s\in D|l(s)=k(a(s)+1)\}|}=\\ =\frac{1-q}{1-q^{d_{\eta-k+1}+1-d_{\eta+1}}}\prod_{i\ge\eta+1}G(d_i-d_{i+1}+\tau(i),d_{i-k}-d_i)_q.
\end{multline}
Here we adopt the following conventions, $d_i=0$, if $-k\le i\le -1$ and $d_{-k-1}=-1$. 

\begin{remark}
Combinatorial constructions from the paper \cite{Loehr} can be used to prove \eqref{original identity}. However, our constructions are different from them. 
\end{remark}

\section{Bijection}\label{section:bijection}

In this section we show how to encode an element of the set $\{D\in\mD|diag^{1,k}(D)=H\}$ as a sequence of partitions. The main result of this section is Proposition~\ref{proposition:bijection}. In section~\ref{subsection:definition} we define a map $F$ from the set $\{D\in\mD|diag^{1,k}(D)=H\}$ to the set of sequences $(P_0,P_1,\ldots)$, where $P_i$ are Young diagrams. In section~\ref{subsection:main properties} we prove the main properties of the map $F$. In section~\ref{subsection:injectivity} we prove an injectivity of the map $F$ and in section~\ref{subsection:image} we describe the image of $F$.

In this section we fix an arbitrary sequence $H=(d_0,d_1,\ldots)$ of non-negative integers.

\subsection{The definition of the map $F$}\label{subsection:definition}

For a Young diagram $D$ let 
\begin{align*}
&B_m(D)=\{j\in\Z_{\ge 0}|r_j(D)\ne 0,kj+r_j(D)-1=m\},\\
&h_m(D)=|\{s=(i,j)\in D|j=m,l(s)=k(a(s)+1)\}|.
\end{align*} 
Let $B_m(D)=\{j_1,j_2,\ldots\}$, where $j_1\le j_2\le\ldots$. Then $h_{j_1}(D)\ge h_{j_2}(D)\ge\ldots$, and we denote the partition $(h_{j_1}(D),h_{j_2}(D),\ldots)$ by $\lambda(D,m)$.

For a partition $\lambda=\lambda_0,\ldots,\lambda_r,\lambda_0\ge\ldots\ge\lambda_r$ let $D_{\lambda}=\{(i,j)\in\Z_{\ge 0}^2|i\le r, j\le \lambda_i-1\}$ be the corresponding Young diagram. Let $\theta(H)$ be the largest $i\le\eta(H)$ such that $i\equiv k-1\mod k$.

Let $D$ be a Young diagram such that $diag^{1,k}(D)=H$. We denote by $F(D)$ a sequence of Young diagrams $(F(D)_0,F(D)_1,\ldots)$ such that $F(D)_i=D_{\lambda(D,i+\theta)}$.

We give an example in Figure~\ref{pic3}. We write the number $i+kj$ into the box $(i,j)\in D$ for the reader's convenience.

\begin{figure}[t]
\begin{center}
\includegraphics{ex3.1}
\end{center}
\caption{}
\label{pic3}
\end{figure}

\subsection{The main properties of $F$}\label{subsection:main properties}

We use the following notations:
\begin{align*}
&w_i(H)=d_{i-k+\theta}-d_{i+\theta}+1,\\
&f_i(H)=\begin{cases}
						d_{i+\theta}-d_{i+1+\theta},&\text{if $k\nmid i$},\\
						d_{i+\theta}-d_{i+1+\theta}+1,&\text{if $k\mid i$}.
				\end{cases}
\end{align*}
We denote by $R(M,N)$ the rectangle in the integral lattice defined by $R(M,N)=\{(i,j)\in\Z_{\ge 0}^2|i\le M-1, j\le N-1\}$. We denote by $\mD(M,N)$ the set $\{D\in\mD|D\subset R(M,N)\}$.

\begin{lemma}
The Young diagram $F(D)_i$ lies in the rectangle $R(f_i,w_i)$.
\end{lemma}
\begin{proof}
Consider a point $(i,j)\in D$. Let $i+kj=l$. Suppose $k\nmid l$, then $(i-1,j)\in D$ and $j\notin B_{l-1}(D)$. Hence, $|B_{l-1}(D)|=d_{l-1}-d_l$. Suppose $k\mid l$. If $i\ne 0$, then $(i-1,j)\in D$ and $j\notin B_{l-1}(D)$. Hence, $|B_{l-1}(D)|\le d_{l-1}-d_l+1$. Thus, we have proved that $r_0(F(D)_{l-1-\theta})\le f_{l-1-\theta}$.

Consider a number $a\in B_l(D)$. Let $d_m'=|\{(i,j)\in D|j\ge a, i+kj=m\}|$. Clearly, $h_a=d_{l-k}'-d_l'+1\le d_{l-k}-d_l+1$. This proves that $c_0(F(D)_{l-\theta})\le w_{l-\theta}$.
\end{proof}
The following statement describes an important property of the numbers $w_i(H)$ and $f_i(H)$. 
\begin{lemma}\label{lemma:not empty}
The set $\{D\in\mD|diag^{1,k}(D)=H\}$ is not empty if and only if for any $i>\eta-\theta$ the following condition holds: $f_i\ge 0,w_i\ge 1$.
\end{lemma}
\begin{proof}
It is easy to check that the set $\{D\in\mD|diag^{1,k}(D)=H\}$ is not empty if and only if for any $i>\eta$ the following three conditions hold: 1) $d_i\le d_{i-k}$; 2) if $k\nmid i$, then $d_i\le d_{i-1}$; 3) if $k\mid i$, then $d_i\le d_{i-1}+1$. These conditions are equivalent to the condition of the lemma.
\end{proof}
Consider a sequence of Young diagrams $P=(P_0,P_1,\ldots)$ such that $P_i\in\mD(f_i,w_i)$ (a short notation for that will be $P\in\prod_{i\ge 0}\mD(f_i,w_i)$). Let $\nu(P)$ be the largest $i$ such that $c_0(P_i)=w_i$. The number $\nu(P)$ is well-defined since $w_0=0$, but it can be equal to $\infty$. It is easy to see that if $P=F(D)$, then $\nu(P)<\infty$.
\begin{lemma}\label{row}
Let $D$ be a Young diagram such that $diag^{1,k}(D)=H$. Then $r_0(D)=\theta(H)+\nu(F(D))+1$.
\end{lemma}
\begin{proof}
Consider a number $a\in B_l(D)$. Suppose that $h_a(D)=d_{l-k}-d_l+1$. Then for any $0\le j\le a$ we have $(r_a(D)-1+kj,a-j)\in D$. In particular, $(0,l)\in D$. Hence $r_0(D)\ge l+1$. On the other hand, $h_0(D)=d_{r_0(D)-1-k}-d_{r_0(D)-1}+1$. This completes the proof of the lemma.
\end{proof}
For a Young diagram $D$ let $D(a,b)=\{(i,j)\in\Z_{\ge 0}^2|(i+a,j+b)\in D\}$. Consider an arbitrary Young diagram $D$ such that $diag^{1,k}(D)=H$. Let $D'=D(0,1),H'=diag^{1,k}(D'),F(D)=(P_0,P_1,\ldots),F(D')=(P_0',P_1',\ldots)$, $f_i'=f_i(H'),w_i'=w_i(H'),\theta'=\theta(H'),\nu=\nu(P),\nu'=\nu(P')$. 
\begin{lemma}\label{change}
$$
d_i'=\begin{cases} 
				d_{i+k}-1, &\text{if $i+k\le\nu+\theta$},\\
     		d_{i+k},   &\text{if $i+k>\nu+\theta$}.
     	\end{cases}
$$
1) If $\nu\ge k$ or $w_k\ge 2$, then 
\begin{align*}
&\theta'=\theta-k; &
&P_i'=\begin{cases}
				P_i,&\text{if $i\ne\nu$},\\
				P_i(1,0),&\text{if $i=\nu$};
      \end{cases}\\
&f_i'=\begin{cases}
				f_i,&\text{if $i\ne\nu$},\\
				f_i-1,&\text{if $i=\nu$};
      \end{cases}&
&w_i'=\begin{cases}
				w_i,&\text{if $i\notin[\nu+1,\nu+k]$},\\
				w_i-1,&\text{if $i\in[\nu+1,\nu+k]$};
      \end{cases}
\end{align*}
2) If $\nu\le k-1$ and $w_k=1$, then 
\begin{align*}
&\theta'=\theta;&
&P_i'=P_i+k;\\
&f_i'=f_{i+k};&
&w_i'=\begin{cases}
				w_i,&\text{if $i>\nu$},\\
				w_i-1,&\text{if $i\le\nu$}.
      \end{cases}
\end{align*}
\end{lemma}
\begin{proof}
The proof is clear from Lemma \ref{row} and the definition of the map $F$.
\end{proof}

\subsection{Injectivity of $F$}\label{subsection:injectivity}

\begin{lemma}\label{injection}
The map $F\colon\{D\in\mD|diag^{1,k}(D)=H\}\to\prod_{i\ge 0}\mD(f_i,w_i)$ is injective.
\end{lemma}
\begin{proof}
The proof is by induction on $|D|$. For $|D|=0$, there is nothing to prove. Assume that $|D|>0$. Using Lemma \ref{change}, we can reconstruct $F(D')$. By the inductive assumption, we can reconstruct $D'$. From Lemma \ref{row} it follows that $F(D)$ determines $r_0(D)$. The diagram $D'$ and the number $r_0(D)$ determines $D$. This completes the proof of the lemma.
\end{proof}

\subsection{The image of $F$}\label{subsection:image}

Consider a sequence $P\in\prod_{i\ge 0}\mD(f_i,w_i)$. For a number $i\ge 0$ let $\Phi_P(i)$ be the minimal $j>i$ such that $r_0(P_j)<f_j$. If for any $j>i$ we have $r_0(P_j)=f_j$, then we put $\Phi_P(i)=\infty$.  
\begin{lemma}\label{condition}
Let $D$ be a Young diagram such that $diag^{1,k}(D)=H$, then for any $i\ge 0$ we have $\Phi_{F(D)}(i)-i\le k$.
\end{lemma}
\begin{proof}
The proof is by induction on $|D|$. For $|D|=0$, there is nothing to prove. Assume that $|D|>0$. We use the notations of Lemma~\ref{change}. Suppose that $\nu>\eta-\theta$ or $\nu=\eta-\theta, f_{\eta-\theta}\ge 2$. From Lemma \ref{change} it follows that for any $i\ge 0$ we have $r_0(P_i)<f_i\Leftrightarrow r_0(P'_i)<f'_i$. Thus, Lemma \ref{condition} follows from the inductive assumption. Assume that $\nu=\eta-\theta$ and $f_{\eta-\theta}=1$. From Lemma \ref{change} it follows that we must only prove that $\Phi_P(\eta-\theta)-(\eta-\theta)\le k$. Assume the converse. Clearly, $w_{\eta-\theta+1}=1$. From the definition of the number $\nu$ and the assumption $\Phi_P(\eta-\theta)-(\eta-\theta)>k$ it follows that $f_{\eta-\theta+1}=0$. Continuing in the same way, we see that $w_{\eta-\theta+1}=w_{\eta-\theta+2}=\ldots=w_{\eta-\theta+k}=1$ and $f_{\eta-\theta+1}=f_{\eta-\theta+2}=\ldots=f_{\eta-\theta+k}=0$. Clearly, $w_{\eta-\theta+k+1}=0$, but this contradicts Lemma \ref{lemma:not empty}. 
\end{proof}
\begin{proposition}\label{proposition:bijection}
Suppose $\{D\in\mD|diag^{1,k}(D)=H\}\ne\emptyset$, then the map
\begin{gather*}
F\colon\{D\in\mD|diag^{1,k}(D)=H\}\to\left\{\left.P\in\prod_{i\ge 0}\mD(f_i,w_i)\right|\begin{smallmatrix}\forall i\ge 0:\\ \Phi_P(i)-i\le k\end{smallmatrix}\right\}.
\end{gather*}
is a bijection such that $|\{s\in D|l(s)=k(a(s)+1)\}|=\sum_{i\ge 0}|F(D)_i|$.
\end{proposition}
\begin{proof}
The second statement of the proposition is clear from the definition of the map $F$. Let us prove that $F$ is a bijection. We have already proved the injectivity. Let us prove the surjectivity of the map $F$. The proof is by induction on $n=\sum_{i\ge 0}d_i$. For $n=0$, there is nothing to prove. Assume that $n\ge 1$. Consider a sequence $P\in\prod_{i\ge 0}\mD(f_i,w_i)$ such that for any $i\ge 0$ we have $\Phi_P(i)-i\le k$. Define $H'$ and $P'$ by formulas from Lemma \ref{change}.

We want to apply the inductive assumption to the sequence $H'$, so we need to check that the set $\{D\in\mD|diag^{1,k}(D)=H'\}$ is not empty. If $\nu=\eta-\theta$, then it easily follows from Lemma \ref{lemma:not empty}. Assume that $\nu>\eta-\theta$. By Lemmas \ref{change} and \ref{lemma:not empty}, we must only prove that for any $\nu<i\le\nu+k$ we have $w_i\ge 2$. Assume the converse. Hence, there exists a number $\nu<i\le\nu+k$ such that $w_i=1$. Therefore, $\sum_{j=1}^k f_{i-j}=1$. Hence, $\Phi_P(i-k-1)=i$. This contradicts the condition $\Phi_P(i-k-1)-(i-k-1)\le k$. Thus, we have prove that $\{D\in\mD|diag^{1,k}=H'\}\ne\emptyset$.  

By the inductive assumption, there exists a Young diagram $D'$ such that $diag^{1,k}(D')=H'$ and $F(D')=P'$. Let us prove that $r_0(D')\le\nu+\theta+1$. By Lemma \ref{row}, it is equivalent to $\nu'+\theta'\le\nu+\theta$ and it follows from Lemma \ref{change}. 

Let $D$ be the diagram obtained from $D'$ by adding the row of length $\nu+\theta+1$. Clearly, $F(D)=P$.   
\end{proof}

\section{Proof of Theorem \ref{main theorem}}\label{section:proof of theorem}

In this section we prove \eqref{combinatorial identity} using Proposition \ref{proposition:bijection}. 

We fix a sequence $H=(d_0,d_1,\ldots)$ such that the set $\{D\in\mD|diag^{1,k}(D)=H\}$ is not empty. We will use the following well known fact (see e.g.~\cite{Andrews})
$$
\sum_{D\in\mD(M,N)}q^{|D|}=G(M,N).
$$
Let $S(H)=\{P\in\prod_{i\ge 0}\mD(f_i,w_i)|\forall i\ge 0:\Phi_P(i)-i\le k\}$. Using Proposition \ref{proposition:bijection} and our notations we see that \eqref{combinatorial identity} is equivalent to the following formula
\begin{gather}\label{main identity}
\sum_{P\in S(H)} q^{|P|}=\frac{1-q}{1-q^{w_{\eta-\theta+1}}}\prod_{i\ge\eta-\theta+1}G(f_i,w_i-1),
\end{gather}
where $|P|=\sum_{i\ge 0}|P_i|$.
Let $\sigma(H)$ be the minimal $i\ge 0$ such that for any $j>\theta+i$ we have $d_j=0$. Let $\psi(H)$ be the maximal $i\le\sigma(H)$ such that $k\mid i$. For a sequence $P\in S(H)$ let $\phi_P(i)$ be the maximal $j<i$ such that $r_0(P_j)<f_j$. We claim that
\begin{gather}\label{extra identity}
\sum_{\substack{P\in S(H)\\\phi_P(\psi+k)=p}}q^{|P|}=q^{\sum_{i=p+1}^{\psi+k-1}f_{i}}\frac{1-q^{f_p}}{1-q^{w_{\psi+k}}}\left(\sum_{P\in S(H)}q^{|P|}\right),
\end{gather}
where $\psi\le p<\psi+k$. 

Let us prove \eqref{main identity} and \eqref{extra identity} by induction on $\sigma$. Suppose $\sigma<k$, then 
\begin{gather*}
	\sum_{P\in S(H)}q^{|P|}=\prod_{i=\eta-\theta+1}^{k-1}G(f_i,w_i)=\frac{1-q}{1-q^{w_{\eta-\theta+1}}}\prod_{i\ge\eta-\theta+1}G(f_i,w_i-1).
\end{gather*}
Hence, \eqref{main identity} is proved. It is clear that
\begin{multline*}
	\sum_{\substack{P\in S(H)\\\phi_P(k)=p}}q^{|P|}=\prod_{i=\eta-\theta+1}^p G(f_i-\delta_i^p,w_i)\prod_{i=p+1}^{k-1}q^{f_i}G(f_i,w_i-1)=\\
	=q^{\sum_{i=p+1}^{k-1}f_i}\frac{1-q^{f_p}}{1-q^{w_k}}\left(\frac{1-q}{1-q^{w_{\eta-\theta+1}}}\prod_{i\ge\eta-\theta+1}G(f_i,w_i-1)\right).
\end{multline*}
Therefore, \eqref{extra identity} is proved.

Suppose $\sigma\ge k$. For $p>\eta(H)$ let
\begin{align*}
&H(p)=(d_0(p),d_1(p),d_2(p),\ldots),\text{where}\\
&d_i(p)=\begin{cases}
					d_{kd_{p+1}+i}-d_{p+1},&\text{if $kd_{p+1}+i\le p$},\\
					0,&\text{if $kd_{p+1}+i> p$}.
				\end{cases}
\end{align*}
If $d_p\ge d_{p+1}$, then $\{D\in\mD|diag^{1,k}(D)=H(p)\}\ne\emptyset$. We adopt the following convention, $S(H(p))=\emptyset$, if $d_p<d_{p+1}$. Note that if $d_p<d_{p+1}$, then $k\mid p+1$. Let $H'=H(\theta+\sigma-1)$ and $H''=H(\theta+\psi-1)$.

Suppose $\psi=\sigma$, then obviously
\begin{gather*}
	\sum_{P\in S(H)}q^{|P|}=\left(\sum_{P'\in S(H')}q^{|P'|}\right)G(f_{\psi}-1,w_{\psi}).
\end{gather*}
By the inductive assumption, the right-hand side is equal to \\
$\frac{1-q}{1-q^{w_{\eta-\theta+1}}} \prod_{i>\eta-\theta} G(f_i,w_i-1)$. 
Suppose $\psi<\sigma$, then
\begin{multline*}
\sum_{P\in S(H)}q^{|P|}=\left(\sum_{P'\in S(H')}q^{|P'|}\right)G(f_{\sigma},w_{\sigma})+\\
+\sum_{p=\sigma-k}^{\psi-1}\left(\sum_{\substack{P''\in S(H'')\\\phi_{P''}(\psi)=p}}q^{|P''|}\right)\left(\prod_{i=\psi}^{\sigma-1}q^{f_i}G(f_i,w_i-1)\right)G(f_{\sigma}-1,w_{\sigma}).
\end{multline*}
By the inductive assumption, the right-hand side is equal to
\begin{multline*}
\frac{1-q}{1-q^{w_{\eta-\theta+1}}}\left[\prod_{i=\eta-\theta+1}^{\sigma-1}G(f_i,w_i-1)\right]\times\\
\times\left(\frac{1-q^{\sum_{i=\psi}^{\sigma-1}f_i}}{1-q}G(f_{\sigma},w_{\sigma})+\frac{1-q^{\sum_{i=\sigma-k}^{\psi-1}f_i}}{1-q}q^{\sum_{i=\psi}^{\sigma-1}f_i}G(f_{\sigma}-1,w_{\sigma})\right).
\end{multline*} 
It is easy to check that it is equal to $\frac{1-q}{1-q^{w_{\eta-\theta+1}}}\prod_{i>\eta-\theta} G(f_i,w_i-1)$.
Hence, \eqref{main identity} is proved.

Let us prove \eqref{extra identity}. Suppose $p>\sigma$, then \eqref{extra identity} is trivial because both sides are equal to zero. Suppose $p<\sigma$, then we have
\begin{gather*}
\sum_{\substack{P\in S(H)\\ \phi_P(\psi+k)=p}}q^{|P|}=\left(\sum_{\substack{P'\in S(H')\\ \phi_{P'}(\psi+k)=p}}q^{|P'|}\right)q^{f_{\sigma}}G(f_{\sigma},w_{\sigma}-1).
\end{gather*} 
By the inductive assumption, the right-hand side is equal to \\
$q^{\sum_{i=p+1}^{\psi+k-1}f_{i}}\frac{1-q^{f_{p}}}{1-q^{w_{\psi+k}}}\left(\sum_{P\in S(H)}q^{|P|}\right)$.

Suppose $p=\sigma$, then we have
\begin{multline*}
\sum_{\substack{P\in S(H)\\ \phi(\psi+k)=\sigma}}q^{|P|}=\left(\sum_{P'\in S(H')}q^{|P'|}\right)G(f_{\sigma}-1,w_{\sigma})+\\
+\sum_{u=\sigma-k}^{\psi-1}\left(\sum_{\substack{P''\in S(H'')\\\phi_{P''}(\psi)=u}}q^{|P''|}\right)\left(\prod_{i=\psi}^{\sigma-1}q^{f_i}G(f_i,w_i-1)\right)G(f_{\sigma}-1,w_{\sigma}).
\end{multline*}
By the inductive assumption, the right-hand side is equal to 
\begin{multline*}
\frac{1-q}{1-q^{w_{\eta-\theta+1}}}\left[\prod_{i=\eta-\theta+1}^{\sigma-1}G(f_i,w_i-1)\right]G(f_{\sigma}-1,w_{\sigma})\times\\
\times\left(\frac{1-q^{\sum_{i=\psi}^{\sigma-1}f_i}}{1-q}+\frac{1-q^{\sum_{i=\sigma-k}^{\psi-1}f_i}}{1-q}q^{\sum_{i=\psi}^{\sigma-1}f_i}\right).
\end{multline*} 
It is easy to check that it is equal to $\frac{1-q^{f_{\sigma}}}{1-q^{w_{\psi+k}}}\left(\sum_{P\in S(H)}q^{|P|}\right)$.
Thus, \eqref{extra identity} is proved. This completes the proof of the theorem.

\section{Proof of Theorem \ref{theorem:q,t-Catalan}}\label{section:proof of Catalan}

We need another description of the varieties $(\C^2)^{[N](k,n)}$. We define the discontinuous map $\rho\colon(\C^2)^{[N]}\to(\C^2)^{[N]}_{1,k}$ by the following formula $\rho(p)=\lim_{t\to 0}tp$, where $p\in (\C^2)^{[N]}$ and $t\in T_{1,k}$. It is easy to see that
\begin{gather*}
(\C^2)^{[N](k,n)}=\rho^{-1}\left(\coprod_{\substack{H=(d_0,d_1,\ldots)\\\sum d_i=N,d_{\ge kn-k}=0}}(\C^2)^{[N]}_{1,k}(H)\right).
\end{gather*}
Clearly, the map $\rho^{-1}\left((\C^2)^{[N]}_{1,k}(H)\right)\xrightarrow{\rho}(\C^2)^{[N]}_{1,k}(H)$ is a locally trivial bundle with an affine space as the fiber. We denote by $d^+_{1,k}(H)$ the dimension of the fiber. Therefore, we have
\begin{gather*}
\left[(\C^2)^{[N](k,n)}\right]=\sum_{\substack{H=(d_0,d_1,\ldots)\\\sum d_i=N,d_{\ge kn-k}=0}}\left[(\C^2)^{[N]}_{1,k}(H)\right]\bL^{d^+_{1,k}(H)}.
\end{gather*} 
Consider the point $p\subset(\C^2)^{[N]}_{1,k}(H)$ corresponding to a monomial ideal $I$. From \eqref{weight decomposition} it follows that
\begin{multline*}
d^+_{1,k}(H)=|\{s\in D_I|l(s)+1>ka(s)\}|+|\{s\in D_I|k(a(s)+1)>l(s)\}|=\\
=|D_I|+|\{s\in D_I|ka(s)\le l(s)<k(a(s)+1)\}|.
\end{multline*}
Obviously, the map $\pi\mapsto D_{\pi}$ is a bijection between the sets $L^+_{kn,n}$ and $\{D\in\mathcal D|diag^{1,k}_{\ge kn-k}(D)=0\}$. Hence, from \eqref{identity1} it follows that   
\begin{multline*}
\sum_{N\ge 0}\left[(\C^2)^{[N](k,n)}\right]t^N=\sum_{\substack{D\in\mathcal D\\diag^{1,k}_{\ge kn-k}(D)=0}}\bL^{|D|+|\{s\in D|ka(s)\le l(s)\le k(a(s)+1)\}|}t^{|D|}=\\
=(\bL t)^{\frac{kn(n-1)}{2}}\sum_{\pi\in L^+_{kn,n}}\bL^{b_k(\pi)}(\bL t)^{-area(\pi)}=(\bL t)^{\frac{kn(n-1)}{2}}C_n^{(k)}(\bL,\bL^{-1}t^{-1}).
\end{multline*} 
This completes the proof of the theorem.

\section{Proof of Theorem \ref{theorem:combinatorial identity}}\label{section:proof of combinatorial identity}

We use the map $\rho\colon(\C^2)^{[N]}\to(\C^2)^{[N]}_{1,k}$ and the numbers $d^+_{1,k}(H)$ from the proof of Theorem~\ref{theorem:q,t-Catalan}. We have
$$
\left[(\C^2)^{[N]}\right]=\sum_{\substack{H=(d_0,d_1,\ldots)\\\sum d_i=N}}\left[(\C^2)^{[N]}_{1,k}(H)\right]\bL^{d^+_{1,k}(H)}.
$$
It is well known (see e.g.\cite{Nakajima}) that
$$
\sum_{N\ge 0}\left[(\C^2)^{[N]}\right]t^N=\prod_{i\ge 1}\frac{1}{1-\bL^{i+1}t^i}.
$$
We know that a sequence $H$ is good if and only if $(\C^2)^{[N]}_{1,k}(H)\ne\emptyset$. The class $\left[(\C^2)^{[N]}_{1,k}(H)\right]$ is computed in Theorem~\ref{main theorem}, so we only need to prove that if $(\C^2)^{[N]}_{1,k}(H)\ne\emptyset$, then 
\begin{gather}\label{formula:plus dimention}
d^+_{1,k}(H)=\sum_{i\ge 0} d_i+\sum_{i\ge\eta}e_i\left(\frac{k}{2}(e_i-1)+\sum_{j=1}^{k-1}(k-j)e_{i+j}\right),
\end{gather}
where $e_i=d_i-d_{i+1}+\tau(i)$. We prove \eqref{formula:plus dimention} by induction on $N$. It is true for $N=0$. Suppose $N\ge 1$. Consider a point $p\subset\left((\C^2)^{[N]}_{1,k}(H)\right)^T$. Let $D$ be the corresponding Young diagram. We have
$$
d^+_{1,k}(H)=|D|+|\{s\in D|ka(s)\le l(s)<k(a(s)+1)\}|.
$$
There exists a unique point $p\in\left((\C^2)^{[N]}_{1,k}(H)\right)^T$ such that the corresponding Young diagram $D$ satisfies the condition $|\{s\in D|l(s)=k(a(s)+1)\}|=0$. It is equivalent to the fact that for any $i\ge 1$ we have $|\{j\in\Z_{\ge 0}|c_j(D)=i\}|\le k$. Let $D'=D(0,1)$ and $H'=diag^{1,k}(D')$. It is easy to see that
\begin{align*}
&|\{s=(i,j)\in D|j=0, ka(s)\le l(s)<k(a(s)+1)\}|=\\
&\sum_{i=0}^{k-1}(d_{\eta-i}-d_{\eta-i+k})=\sum_{i=0}^{k-1}(k-i)e_{\eta+i}-k.
\end{align*}
Therefore, we have
\begin{align*}
&|\{s\in D|ka(s)\le l(s)<k(a(s)+1)\}|=\\
&|\{s\in D'|ka(s)\le l(s)<k(a(s)+1)\}|+\sum_{i=0}^{k-1}(k-i)e_{\eta+i}-k.
\end{align*}
By the inductive assumption, the right-hand side is equal to
\begin{align*}
&\sum_{i\ge\eta}(e_i-\delta_{i,\eta})\left(\frac{k}{2}(e_i-\delta_{i,\eta}-1)+\sum_{j=1}^{k-1}(k-j)e_{i+j}\right)+\sum_{i=0}^{k-1}(k-i)e_{\eta+i}-k=\\
&\sum_{i\ge\eta}e_i\left(\frac{k}{2}(e_i-1)+\sum_{j=1}^{k-1}(k-j)e_{i+j}\right).
\end{align*}
This completes the proof of the theorem.

\section{Homogeneous nested Hilbert schemes}\label{section:homogeneous nested}

In this section we prove Theorem \ref{theorem:nested}. In section \ref{subsection:quiver} we recall the quiver descriptions of the varieties $(\C^2)^{[n]}_{1,k}(H)$ and $(\C^2)^{[\bn]}_{1,1}(\bH)$. In section \ref{subsection:nested proof} we apply this description to conclude the proof of the theorem.

\subsection{A quiver description}\label{subsection:quiver}

The variety $(\C^2)^{[n]}$ has the following description (see e.g.\cite{Nakajima}).
\begin{gather*}
(\C^2)^{[n]}\cong\left.\left\{(B_1,B_2,i)\left|
\begin{smallmatrix}
1) [B_1,B_2]=0\\
2) \text{(stability) There is no subspace} \\
\text{$S\subsetneq\C^n$ such that $B_{\alpha}(S)\subset S$ ($\alpha=1,2$)}\\
\text{and $im(i)\subset S$} 
\end{smallmatrix}\right.\right\}\right/GL_n(\C),
\end{gather*}
where $B_{\alpha}\in End(\C^n)$ and $i\in Hom(\C,\C^n)$ with the action given by $g\cdot (B_1,B_2,i)=(gB_1g^{-1},gB_2g^{-1},gi)$,
for $g\in GL_n(\C)$.

Let $H=(d_0,d_1,\ldots)$. Let $V_i=\C^{d_i}$. It is easy to see that the variety $(\C^2)^{[n]}_{1,k}(H)$ has the following description (see Figure~\ref{quiver1}).
\begin{multline*}
(\C^2)^{[n]}_{1,k}(H)\cong\\
\cong\left.\left\{\left((B_{1,j},B_{2,j})_{j\ge 0},i\right)\left|
\begin{smallmatrix}
1) B_{1,j+k}B_{2,j}-B_{2,j+1}B_{1,j}=0\\
2) \text{There is no graded subspace} \\
\text{$S\subsetneq\bigoplus_{j\ge 0}V_j$ such that $B_{\alpha}(S)\subset S$}\\
\text{($\alpha=1,2$) and $im(i)\subset S$} 
\end{smallmatrix}\right.\right\}\right/\prod_{j\ge 0}GL_{d_j}(\C),
\end{multline*}
where $B_{1,j}\in Hom(V_j,V_{j+1}), B_{2,j}\in Hom(V_j,V_{j+k})$ and $i\in Hom(\C,V_0)$.

\begin{figure}
$$
\xymatrix{
\C\ar[r]^-i & V_0\ar[r]^-{B_2} \ar[d]^-{B_1} & V_k\ar[r]^-{B_2} \ar[d]^-{B_1} & V_{2k}\ar[r]^-{B_2} \ar[d]^-{B_1} & \\
          & V_1\ar[r]^-{B_2} \ar[d]_-{B_1} & V_{k+1}\ar[r]^-{B_2} \ar[d]_-{B_1} & V_{2k+1}\ar[r]^-{B_2} \ar[d]_-{B_1} & \\
          &  \ar@{.}[rrr] \ar[d]_-{B_1} & \ar[d]_-{B_1} & \ar[d]_-{B_1} & \\
          & V_{k-1}\ar[r]^-{B_2} \ar[ruuu]^-(0.45){B_1} & V_{2k-1}\ar[r]^-{B_2} \ar[ruuu]^-(0.45){B_1} & V_{3k-1}\ar[r]^-{B_2} \ar[ruuu]^-(0.45){B_1} & 
}
$$
\caption{The quiver description of $(\C^2)^{[n]}_{1,k}(H)$}
\label{quiver1}
\end{figure}
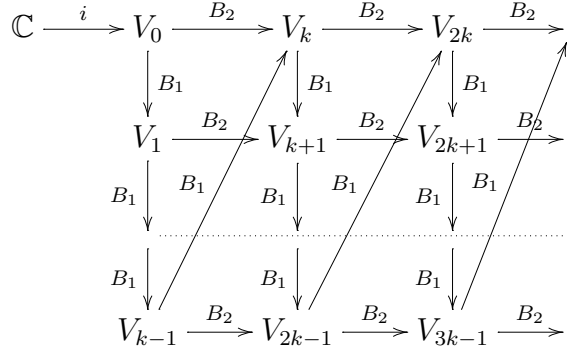

Let $\bH=(H_1,\ldots,H_k)$, where $H_i=(d_{i,0},d_{i,1},\ldots)$. Let $V_{i,j}=\C^{d_{i,j}}$. It is easy to see that the variety $(\C^2)^{[\bn]}_{1,1}(\bH)$ has the following description (see Figure~\ref{quiver2}). 
\begin{multline*}
(\C^2)^{[\bn]}_{1,1}(\bH)\cong \\
\shoveleft{\cong\left\{\left.\left((C_{1,j,h},C_{2,j,h})_{\substack{1\le j\le k\\0\le h}},(p_{j,h})_{\substack{1\le j\le k-1\\0\le h}},i\right)\right|\right.}\\
\left.\left.\begin{smallmatrix}
1) C_{1,j,h+1}C_{2,j,h}-C_{2,j,h+1}C_{1,j,h}=0\\
2) C_{\alpha,j+1,h}p_{j,h}-p_{j,h+1}C_{\alpha,j,h}=0\\
3) \text{There is no graded subspace $S\subsetneq\bigoplus_{j,h}V_{j,h}$} \\
\text{such that $B_{\alpha}(S)\subset S$, $p(S)\subset S$ and $im(i)\subset S$} 
\end{smallmatrix}\right\}\right/\prod_{j,h}GL_{d_{j,h}}(\C),
\end{multline*}
where $C_{\alpha,j,h}\in Hom(V_{j,h},V_{j,h+1})$, $p_{j,h}\in Hom(V_{j,h},V_{j+1,h})$ and $i\in Hom(\C,V_{1,0})$.

\begin{figure}
$$
\xymatrix{
\C\ar[r]^-i & V_{1,0}\ar@<0.5ex>[r]^-{C_1}\ar@<-0.5ex>[r]_-{C_2} \ar[d]^-{p_1} & V_{1,1}\ar@<0.5ex>[r]^-{C_1}\ar@<-0.5ex>[r]_-{C_2} \ar[d]^-{p_1} & V_{1,2}\ar@<0.5ex>[r]^-{C_1}\ar@<-0.5ex>[r]_-{C_2} \ar[d]^-{p_1} & \\
          & V_{2,0}\ar@<0.5ex>[r]^-{C_1}\ar@<-0.5ex>[r]_-{C_2} \ar[d]^-{p_2} & V_{2,1}\ar@<0.5ex>[r]^-{C_1}\ar@<-0.5ex>[r]_-{C_2} \ar[d]_-{p_2} & V_{2,2}\ar@<0.5ex>[r]^-{C_1}\ar@<-0.5ex>[r]_-{C_2} \ar[d]_-{p_2} & \\
          & \ar@{.}[rrr] \ar[d]^-{p_{k-1}} & \ar[d]^-{p_{k-1}} & \ar[d]^-{p_{k-1}} & \\
          & V_{k,0}\ar@<0.5ex>[r]^-{C_1}\ar@<-0.5ex>[r]_-{C_2} & V_{k,1}\ar@<0.5ex>[r]^-{C_1}\ar@<-0.5ex>[r]_-{C_2} & V_{k,2}\ar@<0.5ex>[r]^-{C_1}\ar@<-0.5ex>[r]_-{C_2} & 
}
$$
\caption{The quiver description of $(\C^2)^{[\bn]}_{1,1}(\bH)$}
\label{quiver2}
\end{figure}
\subsection{Proof of Theorem \ref{theorem:nested}}\label{subsection:nested proof}
We use the notations from section \ref{main results:nested}.
\begin{proposition}
There is a natural map $\pi\colon(\C^2)^{[n]}_{1,k}(H)\to (\C^2)^{[\bn]}_{1,1}(\bH)$.
\end{proposition} 
\begin{proof}
Clearly, we have $V_{j,h}=V_{j-1+kh}$, for $1\le j\le k, 0\le h$. We define the map $\pi$ by the following formula $\pi\colon(B_1,B_2,i)\mapsto(C_1,C_2,p,i)$, where $C_1=B_1^k,C_2=B_2,p=B_1$. 
\end{proof}
\begin{proposition}
Under the conditions of Theorem \ref{theorem:nested}, the map $\pi$ is an isomorphism.
\end{proposition}
\begin{proof}
From the stability condition and the commutation relations it follows that the map $p_{j,h}$ is an isomorphism if $d_{j,h}=d_{j+1,h}$. 
Let us define a map $\phi\colon(\C^2)^{[\bn]}_{1,1}(\bH)\to(\C^2)^{[n]}_{1,k}(H)$ by the following formula $\phi\colon(C_1,C_2,p,i)\mapsto(B_1,B_2,i)$, where $B_2=C_2$ and
\begin{gather*}
B_{1,j-1+kh}=\begin{cases}
								p_{j,h},&\text{if $1\le j\le k-1$},\\
								C_{1,1,h}p_{1,h}^{-1}\ldots p_{k-2,h}^{-1}p_{k-1,h}^{-1},&\text{if $j=k$ and $h\in E(\bH)$},\\
								p_{1,h+1}^{-1}\ldots p_{k-2,h+1}^{-1}p_{k-1,h+1}^{-1}C_{1,j,h},&\text{if $j=k$ and $h+1\in E(\bH)$}.
						 \end{cases}
\end{gather*}
Clearly, the map $\phi$ is inverse to $\pi$.
\end{proof}
Theorem \ref{theorem:nested} follows from these two propositions.


\begin{thebibliography}{99}

\bibitem{Andrews} G. E. Andrews. The theory of partitions. Encyclopedia of Mathematics and its Applications, Vol. 2. Addison-Wesley Publishing Co., Reading, Mass.-London-Amsterdam,  1976, 255 pp.

\bibitem{B1} A. Bialynicki-Birula. Some theorems on actions of algebraic groups. Ann. Math. 98, 480-497 (1973).

\bibitem{B2} A. Bialynicki-Birula. Some properties of the decompositions of algebraic varieties determined by actions of a torus. Bull. Acad. Pol. Sci. S6r. Sci. Math. astron. Phys. 24, (No. 9) 667-674 (1976).

\bibitem{Cheah} J. Cheah. Cellular decompositions for nested Hilbert schemes of points.  Pacific J. Math.  183  (1998),  no. 1, 39-90.

\bibitem{Costello} K. Costello and I. Grojnowski. Hilbert schemes, Hecke algebras and the Calogero-Sutherland system. math.AG/0310189.

\bibitem{Ellingsrud} G. Ellingsrud, S. A. Stromme. On the homology of the Hilbert scheme of points in the plane. Invent. math. 87, 343-352 (1987).

\bibitem{Evain} L. Evain. Irreducible components of the equivariant punctual Hilbert schemes.  Adv. Math.  185  (2004),  no. 2, 328-346.

\bibitem{Gusein-Zade} S. M. Gusein-Zade, I. Luengo, A. Melle Hernandez. On generating series of classes of equivariant Hilbert schemes of fat points. Mosc. Math. J. 10 (2010), no. 3.

\bibitem{Gottsche} L. Gottsche. Hilbert schemes of points on surfaces. ICM Proceedings, Vol. II (Beijing, 2002), 483-494.

\bibitem{Haglund} J. Haglund. The $q,t$-Catalan Numbers and the Space of Diagonal Harmonics: With an Appendix on the Combinatorics of Macdonald Polynomials. University of Pennsylvania, Philadelphia - AMS, 2008, 167 pp. 

\bibitem{Iarrobino} A. Iarrobino. Punctual Hilbert schemes. Mem. Am. Math. Soc., (188), 1977.

\bibitem{Iarrobino2} A. Iarrobino, J. Yameogo. The family $G_T$ of graded artinian quotients of $k[x,y]$ of given Hilbert function. Special issue in honor of Steven L. Kleiman.  Comm. Algebra  31  (2003),  no. 8, 3863-3916.

\bibitem{Kirillov} A. N. Kirillov. Combinatorics of Young tableaux and configurations. Proceedings of the St. Petersburg Mathematical Society, Vol. 7, 17–98, Amer. Math. Soc. Transl. Ser. 2, 203, Amer. Math. Soc., Providence, RI, 2001.

\bibitem{Lehn1} M. Lehn. Chern classes of tautological sheaves on Hilbert schemes of points on surfaces. Invent. Math. 136 (1999), no. 1, 157-207.

\bibitem{Lehn2} M. Lehn and C. Sorger. Symmetric groups and the cup product on the cohomology of Hilbert schemes. Duke Math. J. 110 (2001), no. 2, 345-357.

\bibitem{Li} W.-P. Li, Z. Qin, W. Wang, Vertex algebras and the cohomology ring structure of Hilbert schemes of points on surfaces. Math. Ann. 324
(2002), no. 1, 105-133.

\bibitem{Loehr} N. A. Loehr. Conjectured statistics for the higher $q,t$-Catalan sequences. Electron. J. Combin. 12 (2005), Research Paper 9, 54 pp.

\bibitem{Nakajima} H. Nakajima. Lectures on Hilbert schemes of points on surfaces. AMS, Providence, RI, 1999.

\bibitem{Vasserot} E. Vasserot, Sur l'anneau de cohomologie du schema de Hilbert de $\C^2$, C. R. Acad. Sci. Paris S´er. I Math. 332 (2001), no. 1, 7-12.

\end{thebibliography}
\end{document}